\documentclass{amsart}[11pt,A4]
\usepackage[margin=2.54cm]{geometry}

\usepackage{amsmath,amssymb,amsthm,url,xcolor}
\usepackage{hyperref}

\newtheorem{theorem}{Theorem}
\newtheorem{lemma}[theorem]{Lemma}
\newtheorem{corollary}[theorem]{Corollary}

\newtheorem{remark}[theorem]{Remark}

\theoremstyle{definition}

\newcommand{\A}{\mathrm A}
\newcommand{\C}{\mathrm C}

\newcommand{\V}{\mathrm V}
\newcommand{\E}{\mathrm E}
\newcommand{\Z}{\mathrm Z}
\newcommand{\id}{{\rm id}}

\newcommand{\ZZ}{\mathbb{Z}}

\newcommand{\RR}{\mathbb{R}}

\renewcommand{\wr}{\mathop{\rm wr}}

\newcommand{\rank}{\mathrm{rank}}
\newcommand{\fpr}{\mathrm{fpr}}
\newcommand{\fx}{\mathrm{fix}}
\newcommand{\fxr}{\mathrm{rfx}}
\newcommand{\Ker}{\mathrm{Ker}}
\newcommand{\Fix}{\mathrm{Fix}\,}

\newcommand{\Aut}{\mathrm{Aut}}

\newcommand{\Alt}{\mathrm{Alt}}
\newcommand{\Sym}{\mathrm{Sym}}

\begin{document}

\title{On fixity of arc-transitive graphs}

\author[F. Lehner]{Florian Lehner}
\address{Florian Lehner, Institute of Discrete Mathematics, Graz University of Technology, \newline
Steyrergasse 30, 8010 Graz, Austria}\email{f.lehner@tugraz.at}

\author[P. Poto\v{c}nik]{Primo\v{z} Poto\v{c}nik}
\address{Primo\v{z} Poto\v{c}nik,\newline Faculty of Mathematics and Physics,
University of Ljubljana, Slovenia}\email{primoz.potocnik@fmf.uni-lj.si}

\author[P. Spiga]{Pablo Spiga}
\address{Pablo Spiga, Dipartimento di Matematica Pura e Applicata, University of Milano-Bicocca,\newline
 Via Cozzi 55, 20126 Milano, Italy}\email{pablo.spiga@unimib.it}

\thanks{The first author was supported by the Austrian Science Fund (FWF), project W1230-N13}
\thanks{The second author was supported in part by Slovenian Research Agency, programme P1-0294 and project J1-1691.}
\thanks{This article will appear in SCIENCE CHINA Mathematics}

\subjclass[2000]{20B25}
\keywords{permutation group, fixity, minimal degree, graph, automorphism group, vertex-transitive, arc-transitive} 

\maketitle

\begin{abstract}
The relative fixity of a permutation group is the maximum proportion of the points fixed by a non-trivial element of the group and the relative fixity of a graph is the relative fixity of its automorphism group, viewed as a permutation group on the vertex-set of the graph. We prove in this paper that the relative fixity of connected $2$-arc-transitive graphs of a fixed valence tends to $0$ as the number of vertices grows to infinity. We prove the same result for the class of arc-transitive graphs of a fixed prime valence, and more generally, for any class of arc-transitive locally-$L$ graphs, where $L$ is a fixed quasiprimitive graph-restrictive permutation group.
\end{abstract}


\section{Introduction}\label{intro}

For a permutation group $G$ acting on a finite set $\Omega$ and an element $g\in G$,
 let $\Fix_\Omega(g)= \{\omega\in \Omega : \omega^g = \omega\}$ be
 the set of fixed points of $g$, let 
  $$
 \fpr_\Omega(g) := \frac{|\Fix_\Omega(g)|}{|\Omega|}
$$
 be the {\em fixed-point-ratio of $g$}, and let
 $$
 \fx_\Omega(G) := \max\{\Fix_\Omega(g)| : g \in G, g\not =\id_\Omega\}\>\> \hbox{ and }\>\>
 \fxr_\Omega(G) := \max\{\fpr_\Omega(g) : g \in G, g\not =\id_\Omega\} = \frac{\fx_\Omega(G)}{|\Omega|}
 $$
 be the {\em fixity} and the {\em relative fixity} of $G$, respectively.
 
Bounding the fixity of permutation groups
has a long history, going back to a classical result of Jordan 
who proved that for every constant $c$ apart form  a finite list of exceptions (depending on $c$), every primitive permutation group $G\le \Sym(\Omega)$ not containing $\Alt(\Omega)$ satisfies $\fx_\Omega(G) \le |\Omega| - c$.
This result was later improved by several authors, such as Babai \cite{Bab},
Liebeck and Saxl \cite{LS}, Saxl and Shalev \cite{SaxSha}, and Guralnick and Magaard \cite{GurMag}, for example. 
Their work, among other results, amounts to a complete understanding of all primitive groups $G\le \Sym(\Omega)$ with $\fxr_\Omega(G) > 1/2$.
There are several results giving bounds on the fixity of transitive actions of almost simple groups (see, for example, \cite{Tim1,La,LS}), but in general, not much is known about fixity of imprimitive permutation groups.

A natural relaxation of the primitivity condition appears in the theory of
groups acting on graphs. Let $G$ be a transitive
permutation group acting on a finite set $\Omega$ and let $\omega\in \Omega$.
An orbit of the stabiliser $G_\omega$ in the set $\Omega\setminus\{\omega\}$
is then called a {\em suborbit} of $G$. Given a suborbit $\Sigma$
one can construct a so called {\em directed orbital graph} whose vertex-set is $\Omega$
and the set of directed edges is $\{(\omega^g,\sigma^g):g\in G,\sigma\in \Sigma\}$.
If this set of directed edges is invariant under the operation of interchanging the
points in each ordered pair, then the suborbit is called {\em self-paired} and
the directed orbital graph can be viewed as an undirected graph (called simply an {\em orbital graph}) upon which
the group $G$ acts as a group of automorphisms acting transitively on the
{\em arcs} (ordered pairs of adjacent vertices).


A remarkable observation of Donald Gordon Higman \cite{DonHig} asserts that $G$
 acts primitively on $\Omega$ if and only if each of its suborbits
 yields a connected directed orbital graph.
With the existing results on fixity of primitive permutation groups in mind it
is now natural to ask to what extent these results carry over to the case where
at least one directed orbital graph is connected. 
The class of permutation groups having such a suborbit is still too wide for
any meaningful upper bound on the fixity. For example, the imprimitive
wreath product of $C_n\wr\Sym(m)$ acting on the Cartesian product $\Omega$ of a set of size $n$ and a set of size $m$ has a (non-self paired) suborbit of size $m$ yielding
a connected directed orbital graph; namely, the lexicographic product of a directed cycle of length $n$ with an edgeless graph on $m$ vertices, in which the fixity is  $|\Omega| - 2$.

As we shall show in this paper, the situation changes if a suborbit $\Sigma$ corresponding to a stabiliser $G_\omega$ in a transitive permutation group $G\le \Aut(\Omega)$
is self-paired and satisfies additional conditions either on its length
and/or on the permutation group $G_\omega^\Sigma$ induced by the action of $G_\omega$ on $\Sigma$.
The main result of the paper is the following (recall that the socle of a group is the subgroup generated by all minimal normal subgroups):

\begin{theorem}
\label{thm:1}
For every positive real number $\alpha$ there exists a constant $c_\alpha$ with the following property.
Let $G$ be a transitive permutation group acting on a finite set $\Omega$,
$|\Omega| > c_\alpha$,
admitting a self-paired suborbit $\Sigma$ yielding a connected orbital graph, such that at least one of the following holds:
\begin{enumerate}
\item the cardinality of $\Sigma$ is a prime number; 
\item $G_\omega^\Sigma$ is doubly-transitive;
\item $G_\omega^\Sigma$ is primitive with its socle acting regularly
 (i.e., $G_\omega^\Sigma$ is primitive of affine or twisted wreath type).
 \end{enumerate}
Then $\fxr_\Omega(G) < \alpha$.
\end{theorem}

In fact, we prove a slightly more general result, namely Theorem~\ref{the:maindetail}, which is stated in the graph theoretical language in Section~\ref{sec:thm9} (where the proof of Theorem~\ref{thm:1} can also be found). The fixity of the automorphism group of a graph was, to the best of our knowledge, first studied by Babai \cite{Babai1,Babai2} and was motivated by the famous Graph Isomorphism Problem \cite{Babai3}. 
In \cite{ConTuck}, a relationship between the fixity of the automorphism group and the distinguishing number  of a graph was considered.
The fixity of the automorphism group of vertex-transitive graphs of valence at most $4$ was recently investigated in \cite{PotSpiFixicity1} and the present paper can be seen as a strengthening of the results
proved there under additional assumptions on the group of automorphisms.
Finally, we would also like to mention a very interesting line of research \cite{CHKM,KutMar} where the subgraphs induced by the sets of fixed points of automorphisms of cubic arc-transitive graphs were studied.


In what follows, we mostly use standard graph- and group-theoretical notation. In particular, a graph $\Gamma$ is determined by its (finite) vertex-set $\V(\Gamma)$ and edge-set $\E(\Gamma)$ consisting of unordered pairs of (distinct) adjacent vertices. All graphs in this paper are finite and simple. An automorphism of a graph $\Gamma$ is by definition a permutation of $\V(\Gamma)$ which, in its action on unordered pairs of elements of $\V(\Gamma)$, preserves $\E(\Gamma)$. The group of all automorphisms of $\Gamma$ is denoted $\Aut(\Gamma)$. For a vertex $v$ of $\Gamma$,  let $\Gamma(v)$ denote the neighbourhood of $v$. The image of $v$ under $g\in \Aut(\Gamma)$ is denoted by $v^g$.
For every $G\le\Aut(\Gamma)$ there are obvious induced actions of $G$ on $\E(\Gamma)$,
on the {\em arc-set} $\A(\Gamma):=\{(u,v) : \{u,v\} \in \E\Gamma\}$ of $\Gamma$ and on the
set $\A_2(\Gamma):=\{(u,v,w) : \{u,v\}, \{v,w\} \in \E\Gamma, u\not = w\}$ of {\em $2$-arcs} of $\Gamma$.
If $G$ is transitive on $\V(\Gamma)$, $\E(\Gamma)$, $\A(\Gamma)$ or $\A_2(\Gamma)$, then
$\Gamma$ is said to be $G$-vertex-transitive, $G$-edge-transitive, $G$-arc-transitive
or $(G,2)$-arc-transitive, respectively; with the reference to $G$ omitted when $G=\Aut(\Gamma)$.

For a permutation group $G$ on a set $\Omega$, let
 $\Omega/G$ denote the set of all $G$-orbits on $\Omega$
 and let $G^+$ be the group generated by all the  point-stabilisers $G_\omega$, $\omega\in \Omega$.
 Observe that $G^+$ is normal in $G$, implying that
$\Omega/G^+$ is a $G$-invariant partition of $G$. 
For a set $B\subseteq \Omega$, we let $G_B = \{g\in G : B^g=B\}$ be the set-wise stabiliser of
$B$ in $G$.
The centre of a group $G$ will be denoted by $\Z(G)$. If $g, x\in G$ we let $g^x = x^{-1}gx$,
write $g^G=\{g^x : x \in G\}$ and let $\C_G(g) = \{x \in G : g^x=g\}$. 
A permutation group $G$ is said to be {\em quasiprimitive} provided that all non-trivial normal subgroups of $G$ (including $G$ itself) are transitive.


\section{Auxiliary results}

In this section we prove a series of lemmas that are needed in the proof of Theorem~\ref{thm:1}.
Some of them are standard and have appeared elsewhere in a similar form (such as Lemma~\ref{lem:3}), while some are, to the best of our knowledge, new and can be found interesting  on its own (Lemma~\ref{lem:1}, for example).

\begin{lemma}
\label{lem:3}
If $G\le \Sym(\Omega)$ and $g\in G$, then
\begin{equation}
\label{eq:l1r1}
 |\Fix_\Omega(g)| \leq | \C_G(g)| |\Omega/G|.
\end{equation}
Furthermore, if $G$ is normal in a group acting transitively on $\Omega$ and $\omega\in\Omega$, then
\begin{equation}
\label{eq:l1r2}
 \fpr_\Omega(g) \leq \frac{\vert G_\omega \vert \, \vert \C_G(g) \vert}{\vert G \vert}.
\end{equation}
\end{lemma}

\begin{proof}
We prove the first inequality by double counting  the set
\begin{equation*}
S=\{(\delta,x)\mid \delta\in \Omega, x\in G, g^x \in G_\delta\}.
\end{equation*}
By choosing first $x$ and then $\delta$ we see that
\begin{equation}
\label{eq:l1-1}
\vert S \vert =
 \sum_{x \in G} \vert \{ \delta\in \Omega \mid g^x \in G_\delta \}
= \sum_{x \in G} \vert \Fix (g^x) \vert
= \vert \Fix(g) \vert\, \vert G \vert.
\end{equation}
On the other hand,
\begin{equation}
\label{eq:l1-2}
\vert S \vert 
= \sum_{\delta\in \Omega} \vert \{ x \in G \mid g^x \in G_\delta \} \vert.
\end{equation}
Now let $A_\delta:=\{ x \in G \mid g^x \in G_\delta \}$ and let $\varphi \colon A_\delta \to g^G\cap G_\delta$
be a function defined by $\varphi(x):=g^x$. Let $h$ be an arbitrary element of $g^G\cap G_\delta$,
and let $y\in G$ be such that $h=g^y$. Then the preimage $\varphi^{-1}(h)$ consists of elements
$x$ such that $g^x = g^y$, or equivalently, that $xy^{-1} \in \C_G(g)$. Hence $\varphi^{-1}(h) = \C_G(g)y$
and thus $|A_\delta| = |g^G\cap G_\delta|\, |\C_G(g)|$. Using (\ref{eq:l1-1}) and (\ref{eq:l1-2}), it follows that
$$
\vert \Fix(g) \vert\, \vert G \vert = |S|= \sum_{\delta\in \Omega} \vert g^G \cap G_\delta \vert \, \vert \C_G(g) \vert \le |\C_G(g)|  \sum_{\delta\in \Omega} |G_\delta| = |\C_G(g)|\, |\Omega/G|\, |G|,
$$
proving inequality (\ref{eq:l1r1}).
If $G$ is normal in a group acting transitively on $\Omega$, then all of its orbits are of equal length
and hence $|\Omega| = |\omega^G| |\Omega/G| = |G|\, |\Omega/G| / |G_\omega|$. Inequality
(\ref{eq:l1r2}) then follows by dividing inequality (\ref{eq:l1r1}) by $|\Omega|$.
\end{proof}

For a group $X$ and an element $g\in X$ let $g^X$ denote the conjugacy class of $g$ in $X$
and let $\langle g^X\rangle$ be the subgroup of $X$ generated by all the elements of $g^X$.

\begin{lemma}
\label{lem:class}
There exists a strictly decreasing function $f\colon [1,
\infty) \to \RR^+$, $\lim_{x\to \infty} f(x) = 0$,
with the following property: 
If $g$ is an element of a transitive permutation group $X\le \Sym(\Omega)$, $\omega\in \Omega$, and $G = \langle g^X\rangle$, then
$$\fpr_\Omega(g) \le |G_\omega|\, |X:G|\, f(|G:\Z(G)|).$$
\end{lemma}

\begin{proof}
Observe first that $G$ is a normal subgroup of $X$.
Now consider the action of $G$ on the conjugacy class $g^X$ by conjugation. The stabiliser of
 a point $g'\in g^X$ is then the centraliser $\C_G(g')$ and the kernel $K$ of this action
 consists of all the elements of $G$ that centralise every element of $g^X$. 
 Since $G=\langle g^X\rangle$, it follows that $K=\Z(G)$, implying that
 $G/\Z(G)$ acts faithfully on $g^X$ and thus $|G/\Z(G)| \le |g^X|!$.
   On the other hand,
$$
 |g^X| = \frac{|X|}{|\C_X(g)|} =  
 |X:G|\, |G:\C_G(g)|\, \frac{|\C_G(g)|}{ |\C_X(g)|} \le  |X:G|\, |G:\C_G(g)|,
 $$
 showing that 
  $$
  |G:\Z(G)| \le (|X:G|\, |G:\C_G(g)|)!.
 $$
In particular,
by letting $f$ be the function mapping $x\in [1,\infty)$ to $ \frac{1}{\Gamma^{-1}(x-1)}$,
where $\Gamma^{-1}$ is the inverse of the Gamma function restricted to the interval $[2,\infty)$,
we see that $f$ is a strictly decreasing function satisfying
 $$
\frac{|\C_G(g)|}{|G|} \le  |X:G|\, f(|G:\Z(G)|).
 $$
The claim now follows from inequality (\ref{eq:l1r2}) of Lemma~\ref{lem:3}.
\end{proof}


\begin{lemma}
\label{lem:1}
If $G$ is a transitive permutation group acting on a finite set $\Omega$,
then: 
\begin{equation*}
\exp(G)\>\> \hbox{ divides }\>\> |G:\Z(G)|\, |\Omega/G^+|.
\end{equation*}

\end{lemma}

\begin{proof}
Observe first that $G^+$ is a normal subgroup of $G$, implying that $\Omega/G^+$ is a $G$-invariant partition of $\Omega$. 
Let $\omega$ be an arbitrary element of $\Omega$ and let $B=\omega^{G^+}$ be its $G^+$-orbit. 
Note that $(G_B)_\omega = G_\omega \le G^+$.
Since $G^+$, in its action on $B$, is a transitive subgroup of $G_B$, this implies that
$G_B = G^+ (G_B)_\omega = G^+$. Since $G_B$ is the stabiliser of the element $B$ in the
induced action of $G$ on $\Omega/G^+$, this implies that
 the kernel of this action is $G^+$, and that the induced faithful action of $G/G^+$ on $\Omega/G^+$ is semiregular.
 In particular, $|G/G^+|$ divides  $|\Omega/G^+|$.

Since the stabiliser $\Z(G)_\omega$ of a vertex $\omega\in \Omega$ is a normal subgroup of $G$ contained in $G_\omega$, it follows that $\Z(G)_\omega =1$.
Let $t = |G:\Z(G)|$ and let
 \begin{equation*}
 \tau \colon G \to \Z(G), \quad x \mapsto x^t.
\end{equation*} 
Then $\tau$ is a well-defined group homomorphism (see for example \cite[Corollary 7.48]{Rot}).
Let $K = \Ker(\tau)$. Since $|\Z(G)G_\omega/\Z(G)| = |G_\omega|$, we see that the order of $G_\omega$ divides
$t$, and hence $G_\omega$ is a subgroup of $K$, implying that $G^+\le K$. Hence $\exp(N)$ divides $t$. 
The result now follows by the fact that $\exp(G)$ divides $\exp(G^+)\exp(G/G^+)$ and that $|G:G^+|$ divides $|\Omega/G^+|$.
\end{proof}

A group $G$ acts {\em semiregularly} on a set $\Omega$ provided that $G_\omega =1$ for
every $\omega\in \Omega$.
We call the cardinality of a smallest generating set of a group $G$ the {\em rank} of $G$ and
denote it by $\rank(G)$. The following lemma can be proved in many ways and we choose to use the tools from the theory of graph covers as
described in \cite{CovRef} (or see \cite{ElAbCov} for a more succinct explanation of the theory).

\begin{lemma}
\label{lem:cover}
If $\Gamma$ is a connected graph and $G$ a group of automorphisms of $\Gamma$ acting
semiregularly on $\V(\Gamma)$, then $\rank(G) \le |\E(\Gamma)/G| - |\V(\Gamma)/G| +1$. 
\end{lemma}

\begin{proof}
For the purpose of this proof we shall use a more general notion of a graph, namely
one that allows parallel edges, loops and even semiedges (see \cite[Section 3]{CovRef} or \cite[Section 2.1]{ElAbCov} for exact definitions).
Let $\Gamma':=\Gamma/G$ be the quotient graph of $\Gamma$ with respect to $G$
as defined in \cite[Section 2.2]{ElAbCov}.
What follows mimics the classical approach of the theory of covers of topological
spaces with a small modification which is needed due to the possible existence of semiedges in the $\Gamma'$ which arise from the edge-reversing elements in $G$.

 Since $G$ acts semiregularly on $\V(\Gamma)$
the corresponding quotient projection $\wp_G \colon \Gamma \to \Gamma'$ is a
{\em regular covering projection}. The group of covering transformations (which is defined as 
the group of automorphisms of $\Gamma'$ preserving each fibre $\wp_G^{-1}(x)$
where $x$ is either a vertex or a dart of $\Gamma'$) then equals the group $G$.
By the definition of the quotient graph, we have $\V(\Gamma') = \V(\Gamma)/G$
and $\E(\Gamma')=\E(\Gamma)/G$.

Let $\pi(\Gamma',b)$ be the fundamental group based at a vertex
$b$ of $\Gamma$, as defined in \cite[Section 3]{CovRef} (or \cite[Section 2.1]{ElAbCov}). 
Then (see  \cite[Section 3]{CovRef}) $\pi(\Gamma',b)$ is isomorphic to the free product
of $m$ copies of $\ZZ_2$ (where $m$ equals the number of semiedges in $\Gamma'$)
and $\ell$ copies of $\ZZ$. Moreover, $m+\ell$ equals the Betti number of $\Gamma'$,
which equals the number of cotree edges in $\Gamma'$ with respect to an arbitrary spanning
tree of $\Gamma'$. Hence $\rank(\pi(\Gamma';b)) \le m+\ell = |\E(\Gamma')| - |\V(\Gamma')| +1$.

Furthermore, using the procedure described in \cite[Section 2.3]{ElAbCov}, one can find a homomorphism $\zeta\colon \pi(\Gamma',b) \to G$ (called the {\em voltage assignment})
which allows one to reconstruct the graph $\Gamma$ from $\Gamma'$, $G$ and $\zeta$
as the derived covering graph with respect to the {\em locally transitive Cayley voltage space} $(N;\zeta)$. One can easily see that the derived covering graph is connected if and only if
the corresponding homomorphism $\zeta\colon \pi(\Gamma',b) \to G$ is surjective.
Since $\Gamma$ is assumed to be connected, this then implies that
$\rank(G) \le \rank(\pi(\Gamma';b))$ and the result follows.
\end{proof}

If a group of automorphisms $G$ of a graph $\Gamma$ is such that the vertex-stabiliser $G_v$ is transitive on the neighbourhood $\Gamma(v)$ for every
$v\in \V(\Gamma)$, then $\Gamma$ is said to be {\em $G$-locally-arc-transitive}. If $\Gamma$ is a connected $G$-locally-arc-transitive graph, then the group
$G^+$ has index at most $2$ in $G$ and has at most $2$ orbits on $\V(\Gamma)$. If $\Gamma$ is bipartite, then
the orbits of $G^+$ coincide with the parts of the bipartition of $\Gamma$, and if $\Gamma$ is not bipartite, then $G^+=G$ and $G^+$ is arc-transitive.

\begin{corollary}
\label{cor:1}
Let $\Gamma$ be a finite connected $G$-locally arc-transitive graph not isomorphic to a complete bipartite graph, such that $G$ acts faithfully on each of its orbits.
Let $\epsilon = 1$ whenever $\Gamma$ is bipartite and $G$-arc-transitive, and let $\epsilon=0$ otherwise.
Then $\exp(G)$ divides $2^\epsilon |G:Z(G)|$.
\end{corollary}

\begin{proof}
Let $\Omega$ be an orbit of $G$ in its action on $\V(\Gamma)$. 
Suppose first that $\Omega \not = \V(\Gamma)$.
Then $\Gamma$ is bipartite, $\Omega$ is a part of the bipartition of $\Gamma$ and $G=G^+$.
In particular, $\epsilon = 0$ and $|\Omega/G^+| = 1$. By assumption, the action of $G$ on $\Omega$ is faithful and hence $G$
can be viewed as a transitive permutation group of $\Omega$.
By Lemma~\ref{lem:1}, it follows that $\exp(G)$ divides $|G:\Z(G)|$.

Suppose now that $\Omega = \V(\Gamma)$. Then $G$ is arc-transitive and $G^+$ has  at most $2$ orbits on $\Omega$.
Lemma~\ref{lem:1} then yields that $\exp(G)$ divides $2|G:\Z(G)|$. Moreover, if $\epsilon = 0$, then $\Gamma$ not bipartite,
and thus $|\Omega/G^+| = 1$. But then $\exp(G)$ divides $|G:\Z(G)|$, as claimed.
\end{proof}

If $\Gamma$ is a connected graph and $G\le\Aut(\Gamma)$ such that for every vertex $v\in \V(\Gamma)$
the group $G_v^{\Gamma(v)}$ is quasiprimitive (and thus transitive), then we say that $\Gamma$ is $G$-locally quasiprimitive.
Note that such a graph is automatically $G$-locally-arc-transitive.
The following lemma is folklore, but for the sake of completeness we provide the proof.

\begin{lemma}
\label{lem:lqp}
Let $\Gamma$ be a connected $G$-locally quasiprimitive graph. If $G$ acts unfaithfully on one of its
orbits on $\V(\Gamma)$, then $\Gamma$ is a complete bipartite graph.
\end{lemma}

\begin{proof}
Let $\Omega$ be an orbit of $G$ in its action on $\V(\Gamma)$. If the action of $G$ on $\Omega$ is not faithful, then $\V(\Gamma)\not = \Omega$ and hence $\Gamma$ is bipartite and $\Omega$ is one of the two sets of the bipartition with the other set of the bipartition being the second orbit of $G$. Let $K$ be the kernel of the action of 
$G$ on $\Omega$. Since $K\not =1$, there is a vertex $u$ such that $u^K\not = \{u\}$.
Let $v$ be a neighbour of $u$. Since $K$ is a normal subgroup of $G_v$ and since $K$ acts non-trivially on $\Gamma(v)$, it follows that $K$ is transitive on $\Gamma(v)$ and hence $\Gamma(v) = u^K$.
But then $\Gamma(u') = \Gamma(u)$ for every $u'\in \Gamma(v)$. Consequently, the neighbourhood $\Gamma(v')$ of every vertex $v' \in \Gamma(u)$ contains $\Gamma(v)$ and since $v$ and $v'$ are in the
same $G$-orbit, this implies that $\Gamma(v)=\Gamma(v')$. This shows  that every walk starting in $v$ never leaves the set $\Gamma(v) \cup \Gamma(u)$. Since $\Gamma$ is connected, this implies that
$\V(\Gamma) = \Gamma(v) \cup \Gamma(u)$ and thus
$\Gamma$ is complete bipartite.
\end{proof}

\begin{lemma}
\label{lem:4}
There exists an unbounded strictly increasing function $F\colon \RR^+ \to \RR^+$ such that for every connected $G$-locally-quasiprimitive graph $\Gamma$ not isomorphic to a complete bipartite graph the following inequality holds:
$$
|G:\Z(G)| \ge F(|G|).
$$
\end{lemma}

\begin{proof}
Let $\Gamma$ be a connected $G$-locally-quasiprimitive graph not isomorphic to a complete bipartite 
graph and let $Z = \Z(G)$. Since 
\begin{equation}
\label{eq:GZ}
|G| = |G:Z| \, |Z|,
\end{equation}
it suffices to bound $|Z|$ above in terms of $|G:Z|$.
Since $\Gamma$ is not a complete bipartite graph, it follows from Lemma~\ref{lem:lqp} that $G$ acts faithfully on each of its orbits.
By Corollary~\ref{cor:1}, it follows that $\exp(G) \le 2 |G:Z|$, and since $\exp(Z) \le \exp(G)$, we see that
\begin{equation}
\label{eq:exponent}
\exp(Z) \leq 2|G:Z|.
\end{equation} 

We will now establish an upper bound on the rank of $Z$. 
Since $G$ acts faithfully on each of its orbits, the vertex-stabiliser $Z_v$ is trivial for every $v\in \V(\Gamma)$ and thus Lemma~\ref{lem:cover} applies. In particular, 
\begin{equation*}
\label{eq:exponent1}
\rank(Z) \le |\E(\Gamma)/Z| - |\V(\Gamma)/Z| + 1 \le |\E(\Gamma)/Z|.
\end{equation*}
Furthermore,
since $G$ acts transitively on  $\E(\Gamma)$, it follows that $G/Z$ acts transitively on $\E(\Gamma)/Z$, implying that
\begin{equation}
\label{eq:exponent2}
\rank(Z) \le |\E(\Gamma)/Z| \le |G:Z|.
\end{equation}
By combining (\ref{eq:GZ}), (\ref{eq:exponent}) and (\ref{eq:exponent2}),  we thus obtain:
\begin{equation}
\label{eq:exponent3}
|G| = |G:Z| \, |Z| \leq |G:Z| \, \exp(Z)^{\rank (Z)} \leq |G:Z| \, (2 |G:Z|)^{|G:Z|}.
\end{equation}
Let $F$ be the inverse of the (strictly increasing and bijective) function
$\RR^+ \to \RR^+$, 
$x\mapsto x (2x)^x$. The result now follows by applying the function $F$ on
both sides of the inequality (\ref{eq:exponent3}).
\end{proof}

\section{Application to arc-transitive graphs}
\label{sec:thm9}

In this section we formulate and prove the main result of this paper (from which Theorem~\ref{thm:1} follows easily). The formulation of the theorem is rather technical and uses the notion of locally-quasiprimitive group actions on graphs (introduced by Cheryl Praeger in \cite{Pqp}), and the notion of graph-restrictive permutation groups (introduced by Gabriel Verret in \cite{Verret}), which can be defined as follows.

Let $\Gamma$ be a  $G$-vertex-transitive graph.
If the group $G_v^{\,\Gamma(v)}$, induced by the action of the vertex-stabiliser $G_v$ on $\Gamma(v)$, is permutation isomorphic to some permutation group $L$,
 then we say that $G$ is {\em locally-$L$}. 
Similarly, if $G_v^{\,\Gamma(v)}$ is a quasiprimitive permutation group, then we say that $G$ is {\em locally quasiprimitive}.
Following \cite{Verret}, we say that a transitive permutation group $L$ is {\em graph-restrictive} 
provided there exists a constant $c = c(L)$ such that whenever $G$ is an arc-transitive, locally $L$ group 
of automorphisms of a graph $\Gamma$, the order of the stabiliser $G_v$ is at most $c(L)$.

\begin{theorem}
\label{the:maindetail}
For every quasiprimitive and graph-restrictive permutation group $L$
and every positive constant $\alpha$ there exists an integer $N_{L,\alpha}$ with
the following property:
If $\Gamma$ is a connected $X$-arc-transitive graph with $|V(\Gamma)| > N_{L,\alpha}$ and if  $X_v^{\Gamma(v)}$ is permutation isomorphic to $L$ for every vertex $v$,
then
$$
\fpr_{\V(\Gamma)}(g) < \alpha
$$
 for every nontrivial element $g$ of $X$. 
\end{theorem}

\begin{remark}
\label{rem:remark}
{\rm
Determining which transitive permutation groups are graph-restrictive
is a classical topic in algebraic graph theory, going back to William Tutte, who showed in  \cite{tutte} that the symmetric group of degree $3$ is graph-restrictive with the corresponding constant being $48$.
Similarly, it can be deduced from the work of  Anthony Gardiner~\cite{Gard} that the alternating 
group $A_4$ and the symmetric group $S_4$ (both of degree $4$) are
graph restrictive, with corresponding constants $c(A_4) = 36$ and $c(S_4)=2^{4\,}3^{6}$. 
In  \cite[Conjecture 3.12]{weissconj}, Richard Weiss conjectured that every primitive permutation group is graph-restrictive.
Weiss'  conjecture was later strengthened by Cheryl Praeger, conjecturing that every quasiprimitive permutation group is graph-restrictive. Even though both these conjectures are still open, one can deduce from
the work of Richard Weiss and Vladimir Trofimov that
 every doubly transitive group is graph-restrictive. 
The proof of this fact can be found by putting together pieces from many papers, 
but a nice summary is given in the introduction to a later paper by Weiss \cite{dt}.
Together with another result of Richard Weiss \cite{weissp}, this also implies that
every permutation group of prime degree is graph-restrictive. 
In \cite{CGT} and \cite{twist}, the third-named author of this paper proved that every primitive permutation group
of affine type or of twisted wreath type is graph restrictive 
(recall that a primitive permutation group is of affine type provided that it contains a non-trivial abelian normal subgroup and of twisted wreath type if its socle is non-abelian and acts regularly);
in short, primitive permutation group whose socle (group generated by all minimal normal subgroups) acts regularly on the points are graph-restrictive.

Other examples of graph-restrictive groups can be found in \cite{GiuMor1,GiuMor2,Verret,Verret2}, 
and a summary of all (at that time) known graph-restrictive groups  is given in \cite{PSV}. 
However, to deduce Theorem~\ref{thm:1} from Theorem~\ref{the:maindetail}, all that needs to
be remembered is that doubly-transitive permutation groups, primitive permutation groups of affine or twisted wreath type and transitive permutation groups of prime degree are graph-restrictive.
}
\end{remark}


The rest of the section is devoted to proving Theorem~\ref{the:maindetail},
and to a deduction of Theorem~\ref{thm:1} from Theorem~\ref{the:maindetail}.
Let $L$ be a quasiprimitive and graph-restrictive permutation group.
If the degree of $L$ (and thus the valence of $\Gamma$)
is $1$ or $2$, then the result clearly holds. We may thus assume that
the degree of $L$ is at least $3$.

Let $\alpha >0$ and
let $\Gamma$ be a connected $X$-arc-transitive graph 
with $X_v$ permutation isomorphic to $L$, satisfying
\begin{equation}
\label{eq:assump}
\fpr_{\V(\Gamma)}(g)\ge \alpha
\end{equation}
 for some $g\in X \setminus \{1_X\}$.
 Let $c:=c_L$ be the constant associated with
graph-restrictive group $L$; then 
%
$|X_v|\le c.$
%

We need to show that $|\V(\Gamma)|$ is bounded above
by some constant $N$ depending only on $L$ and $\alpha$.
Without loss of generality we may assume that $\Gamma$ is
not a complete bipartite graph and moreover that $|\V(\Gamma)| \ge 1/\alpha$
which together with the assumption  $\fpr_{\V(\Gamma)}(g)\ge \alpha$ implies that 
$g$ fixes at least one vertex of $\Gamma$. Since $\Gamma$ is connected and $g$ is a nontrivial automorphism, it then follows that 
there exists a vertex $v\in \V(\Gamma)$ fixed by $g$, such
that $g$ acts nontrivially on the neighbourhood $\Gamma(v)$.

Let $G = \langle g^X\rangle$. Then $G$ is normal in $X$, implying
that $G_v$ is a normal subgroup of $X_v$. Since $g\in G_v$, it follows that
$G_v$ acts non-trivially on $\Gamma(v)$. Since $X_v^{\Gamma(v)}$ is
quasimprimitive, this implies that $G_v$ acts transitively on $\Gamma(v)$.
Moreover, since $G$ is normal in a vertex-transitive group $X$, it follows that
$G_u$ is transitive on $\Gamma(u)$ for every $u\in \V(\Gamma)$; that is, $\Gamma$
is $G$-locally arc-transitive. 
Since $G\le X$, it follows that 
\begin{equation}
\label{eq:proof01}
|G_v|\le c.
\end{equation}

 Moreover, since $G$ has at most $2$ orbits on $\V(\Gamma)$, it follows that 
\begin{equation}
\label{eq:proof02}
|X:G| = \frac{|\V(\Gamma)| |X_v|}{|v^G| |G_v|} \le 2 |X_v:G_v| \le |X_v|\le c.
\end{equation}

By Lemma~\ref{lem:4} we also see that
\begin{equation}
\label{eq:proof1}
|G:\Z(G)| \ge	F(|G|).
\end{equation}
for some fixed unbounded strictly increasing function $F\colon \RR^+ \to \RR^+$,
and by Lemma~\ref{lem:class} it follows that
\begin{equation}
\label{eq:proof2}
\fpr_{\V(\Gamma)}(g) \le |G_\omega|\, |X:G|\, f(|G:\Z(G)|)
\end{equation}
for some fixed strictly decreasing function $f\colon [1,\infty) \to \RR^+$ such that
$\lim_{x\to \infty} f(x) = 0$. Combining inequalities (\ref{eq:proof01}), (\ref{eq:proof02}), (\ref{eq:proof1}) and
(\ref{eq:proof2}), we see that
$$
\alpha\le \fpr_{\V(\Gamma)}(g) \le c^2 \varphi(|G|)
$$
where $\varphi := f \circ F$
is a strictly decreasing function such that $\lim_{x\to \infty} \varphi(x) = 0$.
By dividing by $c^2$ and applying the inverse of $\varphi$, one thus concludes that
$$
 |G| \le \varphi^{-1}(\alpha/c^2).
$$
Since $|\V(\Gamma)| \le 2 |G|/|G_v| \le |G|$ this yields an upper bound 
$N_{L,\alpha}:=\varphi^{-1}(\alpha/c^2)$ for $|\V(\Gamma)|$
which depends only on $\alpha$ and $L$. In particular, if $|V(\Gamma)|>N_{L,\alpha}$,
then the assumption (\ref{eq:assump}) must be false. Hence $|V(\Gamma)|\le N_{L,\alpha}$.
This finishes the proof of Theorem~\ref{the:maindetail}.
\smallskip

Theorem~\ref{thm:1} now follows easily from Theorem~\ref{the:maindetail} and Remark~\ref{rem:remark}. Indeed, let $\alpha$ be a positive constant, let $G$ be a transitive permutation group acting on a finite set $\Omega$,
let $\omega\in \Omega$ and let $\Sigma = \delta^{G_\omega}$ be a
 self-paired suborbit  yielding a connected orbital graph $\Gamma$.
Then $\V(\Gamma)  = \Omega$, $\E(\Gamma) = \{\{\omega^g, \delta^g\} :
g\in G\}$ and $\Gamma$ is a $G$-arc-transitive graph of valence $|\Sigma|$.
Suppose in addition that either $|\Sigma|$ is a prime number or that
the permutation group $G_\omega^{\Sigma}$
induced by the action of $G_\omega$ on $\Sigma$
 is doubly-transitive or primitive of affine type.
  Then $L:=G_\omega^{\Sigma}$ is clearly a primitive permutation group
 and in view of Remark~\ref{rem:remark}, it is also graph-restrictive.
By Theorem~\ref{the:maindetail}, there exists a constant $c_\alpha:=N_{L,\alpha}$ 
such that $\fpr_{\V(\Gamma)}(g) < \alpha$ for every $g\in G\setminus\{1\}$.
In particular, $\fxr_\Omega(G) \le \alpha$, thus proving Theorem~\ref{thm:1}.

\end{document}